\documentclass{article}
\usepackage{amsmath}
\usepackage{amssymb}
\usepackage{amsthm} 
\usepackage[utf8]{inputenc}
\usepackage{cleveref}
\usepackage{blindtext}
\usepackage{titlesec}
\title{Generalized Harnack Inequality for Mean Curvature Flow and Ancient Solutions}
\author{Junyoung Park }
\date{}
\newtheorem{theorem}{Theorem}[section]
\newtheorem{lemma}{Lemma}[section]
\newtheorem{corollary}{Corollary}[section]
\newtheorem{remark}{Remark}[section]
\usepackage{geometry}
\begin{document}

\maketitle
\abstract The goal of this paper is to relax convexity assumption on some classical results in mean curvature flow. In the first half of the paper, we prove a generalized version of Hamilton's differential Harnack inequality which holds for mean convex solutions to mean curvature flow with a lower bound on $\frac{\lambda_1}{H}$ where $\lambda_1$ is the smallest principal curvature. Then, we use classical maximum principle to provide several characterizations of family of shrinking spheres for closed, mean convex ancient 
solution to mean curvature flow with a lower bound on $\frac{\lambda_1 + .. + \lambda_k}{H}$ for some $1 \leq k \leq d-1$, where $\lambda_1 \leq \lambda_2 \leq .. \leq \lambda_d$ are the principal curvatures.
\section{Introduction}
Let $M^d$ be a smooth $d$-dimensional manifold. A smooth one-parameter family of hypersurfaces $\{M_t^d\}_{t \in I}$ is a smooth solution to mean curvature flow if there exists a smooth one parameter family of immersions $X_t : M^d \to \mathbb{R}^{d+1}$ which satisfies the equation
$$\frac{\partial X_t}{\partial t}(p,t) = \Vec{H}(p,t)$$
where $\Vec{H}$ is the mean curvature vector. This geometric equation has been studied extensively in the past several decades to see if one can evolve a given hypersurface to a relatively simpler one. Unfortunately, in many cases, the solution develops singularity in finite time due to the nonlinearity of the equation. Therefore, it is crucial to have a detailed understanding of the nature of the singularities, so that one might come up with a systematic way to flow through the singularities. \\

A standard method of analyzing singularities is to use blowup arguments. As the solution approaches singular time, we `magnify' at larger and larger scale near the point where singularities form, and take `the limit' of the rescaled solutions. In the end, one is left with a solution to the mean curvature flow which is defined for all negative time. Such solutions are called ancient solutions. Therefore, the study of ancient solutions is important since they model how singularities form. \\

In this paper, we are concerned with smooth solutions to mean curvature flow which are mean convex ($H > 0$) and have a global lower bound on
$$\frac{\lambda_1 + .. + \lambda_{k}}{H}$$
for some $1 \leq k \leq d-1$, where $\lambda_1 \leq \lambda_2 \leq .. \leq \lambda_{d-1}$ are the principal curvatures. Our goal is to extend two well-known results which hold for convex solutions to mean curvature flow, namely Hamilton's differential Harnack inequality (\cite{hamilton1995harnack}, theorem 1.1) and characterizations of family of shrinking spheres (\cite{10.4310/jdg/1442364652}, theorem 1.1) to our more general class of solutions.\\

 Hamilton's differential Harnack inequality for convex solutions to mean curvature flow is one of the tools used to analyze ancient solutions to mean curvature flow. Although the Harnack inequality is useful, the convexity condition $A \geq 0$ is a rather strong curvature condition and thus prevents it from being applied to more general situations. In the first half of the paper, we extend the differential Harnack inequality to our more general class of solutions with $k = 1$.
\begin{theorem}
    Let $\{M_t^d\}_{t \in [0, T)}$ be a mean convex solution to mean curvature flow such that it is either closed or complete with bounded curvature and
    $$A + \epsilon_0 Hg \geq 0 \text{ in }M^d \times [0, T)$$
    for some $\epsilon_0 \geq 0$. Let $$\phi(t) = \sup_{M^d \times (0, t]}H < \infty$$ Then
    $$Z(V) = (A + \epsilon_0Hg)(V,V) + \frac{2}{1 + \epsilon_0d}\nabla H\cdot V + \frac{H}{2t} + \frac{1}{(1 + \epsilon_0 d)^2}\nabla_tH + \frac{\epsilon_0}{(1 + \epsilon_0d)^2}H^3 + \frac{3\epsilon_0}{4(1 + \epsilon_0d)}\phi^2(t)H \geq 0$$
    for all $V \in TM^d$, $t \in (0, T)$. 
\end{theorem}
In the weakly convex case, we can set $\epsilon_0 = 0$, thus recovering the classical Harnack inequality.
\begin{corollary}[Recovery of classical Harnack inequality]
    If $\{M_t^d\}_{t \in [0, T)}$ is weakly convex solution to mean curvature flow which is either closed or complete with bounded curvature, 
    $$Z(V) = A(V,V) + 2\nabla H\cdot V + \frac{H}{2t} + \nabla_tH  \geq 0$$
    for $V \in TM^d$, $t \in (0, T)$.
\end{corollary}
In the second half of the paper, we consider ancient solutions to mean curvature flow. In \cite{10.4310/jdg/1442364652}, Huisken and Sinestrari provided several characterizations of shrinking spheres for closed, convex ancient solution to mean curvature flow. 
\begin{theorem}[cf. \cite{10.4310/jdg/1442364652}, theorem 1.1]
Let $\{M_t^d\}_{t \in (-\infty, 0)}$ be a closed, convex ancient solution to mean curvature flow of dimension $d \geq 2$. Then the following are equivalent. \\\\
(1) $\{M_t^d\}_{t \in (-\infty, 0)}$ is a family of shrinking $d$-dimensional spheres.\\
(2) $\liminf_{t \to -\infty}\min_{M^d\times \{t\}}\frac{\lambda_1}{H} > 0$ i.e $\{M_t^d\}_{t \in (-\infty, 0)}$ is uniformly strictly convex.\\
(3) $\limsup_{t \to -\infty}\frac{\textit{{\normalfont diam}}(M_t^d)}{\sqrt{-t}} < \infty$. \\
(4) There exists $C > 0$ so that $\max_{M^d \times \{t\}}H \leq C\min_{M^d \times \{t\}}H$
for all sufficiently small $t$.\\
(5) $\limsup_{t \to -\infty}\sqrt{-t}\max_{M^d \times \{t\}}H < \infty$\\
 (6) $\limsup_{t \to - \infty}\frac{\rho_+(t)}{\rho_-(t)} < \infty$
 where $\rho_+(t)$ is the circumradius of $M_t^d$ and $\rho_-(t)$ is the inradius of $M_t^d$.\\
 (7) $\limsup_{t \to -\infty}\frac{|M_t^d|^{d+1}}{|\Omega_t|^d} < \infty$ where $\partial \Omega_t = M_t^d$
\end{theorem}
As in the first half of the paper, our objective is to relax the convexity assumption and obtain an analogous characterization result for closed, mean convex ancient solutions of $d \geq 2$ which have a global lower bound on
$$\frac{\lambda_1 + .. + \lambda_k}{H}$$
for some $1 \leq k \leq d-1$.
Notice that unlike in the convex case, the solution is not necessarily embedded. Since we do not wish to assume that our solution is embedded, we do not know if there exists an open domain  $\Omega_t \subset \mathbb{R}^{d+1}$ so that $\partial \Omega_t = M_t^d$. Thus, we ignore condition (6) and (7), and focus our attention to the first five conditions instead.
\begin{theorem}
 Let $\{M^d_t\}_{t \in (-\infty, 0)}$ be a closed, mean convex, ancient solution to mean curvature flow of dimension $d \geq 2$ such that
 $$\liminf_{t \to -\infty}\min_{M^d \times \{t\}}\frac{\lambda_1 + .. + \lambda_{k}}{H}> -\infty$$
 for some $1 \leq k \leq d-1$ and $$\limsup_{t \to -\infty}\max_{M^d \times \{t\}}H< \infty$$ Then the following are equivalent.\\\\
 (1) $\{M_t^d\}_{t \in (-\infty, 0)}$ is a family of shrinking $d$-dimensional spheres.\\\\
 (2) $$\liminf_{t \to -\infty}\min_{M^d\times \{t\}}\frac{\lambda_1}{H} > 0$$ i.e $\{M_t^d\}_{t \in (-\infty, 0)}$ is uniformly strictly convex.\\\\
 (3)$$\limsup_{t \to -\infty} \max_{M^d \times \{t\}}H\textit{{\normalfont diam}}(M^d_t) < \infty$$
 where $\textit{{\normalfont diam}}(M^d_t)$ denote the diameter of $M_t^d \subset \mathbb{R}^{d+1}$.\\\\
 (4) There exists $C > 0$ so that $\max_{M^d \times \{t\}}H \leq C\min_{M^d \times \{t\}}H$ for all sufficiently small $t$.\\\\
 (5) $\limsup_{t \to -\infty}\sqrt{-t}\max_{M^d \times \{t\}}H < \infty$.
\end{theorem}
We first remark that in theorem 1.3, we additionally assume an upper bound on $H$ as $t \to -\infty$. This condition is actually superfluous in the convex case, since the upper bound on $H$ automatically follows from the classical Harnack inequality (see \cite{10.4310/jdg/1442364652} for details). We also note that while the scaling factor in condition (3) of theorem 1.2 is $\frac{1}{\sqrt{-t}}$, it is $\max_{M^d \times \{t\}}H$ in condition (3) of theorem 1.3. The latter is in general a stronger hypothesis because one always has
$$\frac{C}{\sqrt{-t}} \leq \max_{M^d \times \{t\}}H$$
for some fixed constant $C > 0$ for all sufficiently small $t$ by the weak maximum principle. Also, similar to the previous remark, the two conditions on the diameter are actually equivalent in the convex case due to the classical Harnack inequality (see \cite{10.4310/jdg/1442364652}, lemma 4.1). \\

There are other works generalizing \cite{10.4310/jdg/1442364652}, theorem 1.1 to a more general class of ancient solutions to mean curvature flow such as \cite{haslhofer2016ancient}, theorem 1.5 and \cite{langford2017general}, corollary 1.5, 1.6. In the work of Haslhofer and Hershkovits, they assume the ancient solution to be $\alpha$-noncollapsed for some $\alpha > 0$. In our case, the ancient solutions are not necessarily embedded, thus we work with solutions which are not necessarily noncollapsed. In the work of Langford, he assumes the ancient solution to be closed, have bounded rescaled volume and lower bound on $\frac{\lambda_1}{H}$. In our case, the curvature assumption is more general since one can take any $1 \leq k \leq d-1$. Also, bounded rescaled volume does not necessarily follow from bounded mean curvature in general since the proof of \cite{langford2017general}, lemma 5.2 requires the solution to be embedded. \\

We mention that the proof only uses classical weak and strong maximum principle for tensors, and does not require any additional heavy machinery. This fact motivates us to apply the argument used to prove theorem 1.3 to noncompact solutions.
\begin{theorem}
    Let $\{M_t^d\}_{t \in (-\infty, 0)}$ be complete, mean convex, ancient solution to mean curvature flow of $d \geq 2$ with $\limsup_{t \to -\infty}\sup_{M^d \times \{t\}}H < \infty$. Assume 
    $$\liminf_{t \to -\infty}\inf_{M^d \times \{t\}}\frac{\lambda_1 + .. + \lambda_{k}}{H} > -\infty $$
    for either $k = 1$ or $k = d-1$. We also impose a `flatness at low curvature region' condition, that is for every $\epsilon > 0$, there exists $\delta > 0$ so that whenever 
    $$\frac{H(p,t)}{\phi(t)} < \delta$$
    then
    $$\frac{\lambda_1 + .. + \lambda_{k}}{H}(p,t) > -\epsilon $$
    where $\phi(t) = \sup_{M^d \times (-\infty, t]}H$. \\\\
    Then the solution is weakly $k$ convex, i.e
   $$\liminf_{t \to -\infty}\inf_{M^d \times \{t\}}\frac{\lambda_1 + .. + \lambda_{k}}{H} \geq 0 $$
\end{theorem}
\subsection*{Acknowledgment}
The author would like to thank his advisor Prof. Natasa Sesum for her various suggestions and comments to improve the results of this paper.
\section{Generalized Harnack inequality}
\begin{proof}[Proof of theorem 1.1]
 For simplicity, we define
$$\delta = \frac{1}{1 + \epsilon d}, \ \ f(t) = \frac{1}{2t} + \frac{3}{2}\epsilon \delta \phi^2(\hat{T}), \ \ g(t) = \frac{1}{2t} + \frac{3\epsilon \delta}{4}\phi^2(\hat{T})$$
We use indices $a,b,c$, and $\nabla_t$ to denote covariant time derivative as in \cite{hamilton1995harnack}. 
We show that for any $0 < t \leq \hat{T} < T$, $\epsilon > \epsilon_0$
\begin{align*}
Z_{\epsilon, \hat{T}}(V) &= (A + \epsilon Hg)(V,V) + \frac{2}{1 + \epsilon d}\nabla H\cdot V + \frac{1}{(1 + \epsilon d)^2}\nabla_tH+ \frac{H}{2t}   + \frac{3\epsilon}{4(1 + \epsilon d)}\phi(\hat{T})^2H+ \frac{\epsilon}{(1 + \epsilon d)^2}H^3 \\ &= (A + \epsilon Hg)(V, V) + 2\delta \nabla H\cdot V + \delta^2\nabla_tH + g(t)H + \epsilon \delta^2 H^3 \geq 0    
\end{align*}
for all $V \in TM^d$, $0 < t \leq \hat{T} < T$. 
Then theorem 1.1 follows by first letting $\epsilon \to \epsilon_0$, then
$t = \hat{T} \in (0, T)$. 
Since $\{M^d_t\}_{t \in [0, T)}$ is a mean convex solution to mean curvature flow, by the strong maximum principle either $H > 0$ for all $t \in (0, T)$ or $H = 0$. The theorem trivially holds when $H = 0$, thus WLOG we assume $H > 0$ for $t > 0$. Then for any $\epsilon > \epsilon_0$, $$A + \epsilon Hg > 0 $$
i.e it is invertible. This ensures that at each $(p,t) \in M^d \times (0, \hat{T}]$, there exists a unique $V(p,t) \in T_pM^d$ which minimizes
$$T_pM^d \ni V \to Z_{\epsilon, \hat{T}}(V)$$
and is characterized by
$$A_{ab}V_b + \epsilon HV_a + \delta \nabla_aH = 0$$
which is the 1st variation formula w.r.t $V$. \\

We claim that at each $(p_0,t_0) \in M^d \times (0, \hat{T}]$, we can find a local extension of $V(p_0,t_0)$ in a small backward parabolic cylinder centered at $(p_0,t_0)$ denoted also by $V$ so that
$$(\nabla_t - \Delta)|_{(p_0,t_0)}Z_{\epsilon, \hat{T}}(V) \geq (|A|^2 - \frac{4f}{\delta})Z_{\epsilon, \hat{T}}(V)$$
For each $(p_0,t_0) \in M^d \times (0, \hat{T}]$, we let $(U, \{x_a\})$ to be a geodesic normal coordinates w.r.t $g(t_0)$ centered at $p_0$ so that the coordinate vectors at $p_0$ are the eigenvectors of the second fundamental form i.e at $(p_0, t_0)$,
$$A_{ab} = \lambda_a\delta_{ab}$$
We define our local orthonormal frame $e_a(p,t)$ by first letting $e_a(p_0, t_0) = \frac{\partial}{\partial x^a}|_{p_0}$, then extending it to $U \times \{t_0\}$ by parallel transport (w.r.t $g(t_0)$) along radial geodesics emanating from $p_0$ (w.r.t $g(t_0)$), then extending it to $U \times (t_0 - \eta, t_0]$ by $\nabla_t$-parallel transport for some small $\eta > 0$. We then define our local extension to be
\begin{align*}
    U \times (t_0 - \eta, t_0] \ni (p,t) \to V(p,t) =& \ (V_a(p_0, t_0) + (\delta H(p_0, t_0)\lambda_a + f(t_0))x_a(p) \\ &+ (-\frac{2f(t_0)}{\delta}V_a(p_0, t_0) - \delta \lambda_a\nabla_aH(p_0, t_0))(t - t_0))e_a(p,t)
\end{align*}
where $x_a(p)$ is the $a$th coordinate of $p$. Then $V$ is indeed a local extension of $V(p_0,t_0)$ and we have
\begin{align*}
   (\nabla_t - \Delta)|_{(p_0, t_0)}V_a &= -\frac{2f(t_0)}{\delta}V_a(p_0, t_0) - \delta \lambda_a\nabla_aH(p_0, t_0) = -\delta A_{ab}\nabla_bH - \frac{2f(t_0)}{\delta}V_a \\ \nabla_aV_b|_{(p_0, t_0)} &= \delta H(p_0, t_0)\lambda_a\delta_{ab} + f(t_0)\delta_{ab} = \delta HA_{ab} + f g_{ab}
\end{align*}
We then calculate $(\nabla_t - \Delta)|_{(p_0,t_0)}Z_{\epsilon, \hat{T}}(V)$ using above extension. \\

We calculate the evolution equation using computations given in \cite{hamilton1995harnack}. Note that all computations below are done at $(p_0, t_0)$ w.r.t our choice of local extension of $V(p_0, t_0)$.
\begin{align*}
    (\nabla_t - \Delta)Z_{\epsilon, \hat{T}}(V) =& |A|^2Z_{\epsilon, \hat{T}}(V) + (2A_{ab}V_b + 2\epsilon Hg_{ab}V_b + 2\delta\nabla_aH)(\nabla_t - \Delta)V_a -2(A_{ab} + \epsilon Hg_{ab})\nabla_cV_a\nabla_cV_b \\ &+(-4\nabla_cA_{ab}V_b - 4\epsilon\nabla_cHg_{ab}V_b - 4\delta\nabla_tA_{ac} + 4\delta HA_{ab}A_{bc})\nabla_cV_a + 4\delta H\nabla_aA_{bc}A_{bc}V_a 
 \\ &+ 2\delta A_{ab}A_{bc}\nabla_cHV_a + 4\delta^2HA_{ab}\nabla_tA_{ab} +2\delta^2A_{ab}\nabla_aH\nabla_bH - 2\delta^2H^2A_{ab}A_{bc}A_{ca} \\ &+ 2\epsilon\delta^2|A|^2H^3 + g'(t)H - 6\epsilon\delta^2H|\nabla H|^2
\end{align*}
We compute each terms separately according to our choice of local extension.
\begin{align*}
    (2A_{ab} + 2\epsilon Hg_{ab})(\nabla_t - \Delta)V_aV_b =& -2\delta A_{ab}A_{ac}V_b\nabla_cH - 2\epsilon \delta HA_{ab}V_a\nabla_bH - 2\delta^2A_{ab}\nabla_aH\nabla_bH \\ &- \frac{4f}{\delta}(A_{ab}+ \epsilon Hg_{ab})V_aV_b - 4f\nabla_aHV_a 
\end{align*}
\begin{align*}
   -2(A_{ab} + \epsilon Hg_{ab})\nabla_cV_a\nabla_cV_b &= -2(A_{ab} + \epsilon Hg_{ab})(\delta HA_{ac} + fg_{ac})(\delta H A_{bc} + fg_{bc}) \\ &=-2(A_{ab} + \epsilon Hg_{ab})(\delta^2H^2A_{ac}A_{bc} + 2f\delta HA_{ab} + f^2g_{ab}) \\ &= -2\delta^2H^2A_{ab}A_{bc}A_{ca} - 4f\delta H|A|^2 - 2f^2H-2\epsilon\delta^2H^3|A|^2-4f\epsilon\delta H^3-2\epsilon df^2H
\end{align*}
\begin{align*}
  &(-4\nabla_cA_{ab}V_b - 4\epsilon\nabla_cHg_{ab}V_b - 4\delta\nabla_tA_{ac} + 4\delta HA_{ab}A_{bc})\nabla_cV_a \\ = &(-4\nabla_bA_{ac}V_b - 4\epsilon\nabla_cHg_{ab}V_b - 4\delta\nabla_tA_{ac} + 4\delta HA_{ab}A_{bc})(\delta HA_{ac} + fg_{ac}) \\ =& -4\delta HA_{ac}\nabla_bA_{ac}V_b - 4f\nabla_aHV_a-4\epsilon\delta HA_{ab}V_b\nabla_aH - 4\epsilon f\nabla_aHV_a \\ &-4\delta^2H\nabla_tA_{ab}A_{ab} - 4\delta f\nabla_tH + 4\delta^2H^2A_{ab}A_{bc}A_{ca} + 4f\delta H|A|^2
\end{align*}
Combining all the terms, we obtain
\begin{align*}
   (\nabla_t - \Delta)Z_{\epsilon, \hat{T}}(V) =& |A|^2Z_{\epsilon, \hat{T}}(V) - \frac{4f}{\delta}(A_{ab} + \epsilon Hg_{ab})V_aV_b - 8f\nabla_aHV_a - 6\epsilon \delta HA_{ab}V_b\nabla_aH - 4\epsilon f\nabla_aHV_a \\ &- 4\delta f \nabla_tH + g'(t)H -6\epsilon \delta^2H|\nabla H|^2 - 2f^2H - 4f\epsilon\delta H^3 - 2\epsilon df^2H 
\end{align*} 
Since $V(p_0, t_0)$ is characterized by the 1st variation formula
$$A_{ab}V_b + \epsilon HV_a + \delta \nabla_aH = 0$$
we have
$$- 6\epsilon \delta HA_{ab}V_b\nabla_aH = 6\epsilon\delta H(\epsilon HV_a + \delta \nabla_aH)\nabla_aH = 6\epsilon^2\delta H^2\nabla_aH V_a + 6\epsilon \delta^2H|\nabla H|^2$$
This gives us 
\begin{align*}
(\nabla_t - \Delta)Z_{\epsilon, \hat{T}}(V) =& |A|^2Z_{\epsilon, \hat{T}}(V) - \frac{4f}{\delta}[(A_{ab} + \epsilon Hg_{ab})V_aV_b +2\delta \nabla_aHV_a + \delta^2\nabla_tH] \\ &+ (6\epsilon^2\delta H^2 - 4\epsilon f)\nabla_aHV_a + g'(t)H - 2f^2H - 4f\epsilon \delta H^3 - 2\epsilon df^2H \\ =& |A|^2Z_{\epsilon, \hat{T}}(V) - \frac{4f}{\delta}[Z_{\epsilon, \hat{T}}(V) - g(t)H - \epsilon\delta^2H^3] \\ &+ (6\epsilon^2\delta H^2 - 4\epsilon f)\nabla_aHV_a+ g'(t)H - 2f^2H - 4f\epsilon \delta H^3 - 2\epsilon df^2H\\ =& (|A|^2 - \frac{4f}{\delta})Z_{\epsilon, \hat{T}}(V) + (6\epsilon^2\delta H^2 - 4\epsilon f)\nabla_aHV_a + (\frac{4fg}{\delta} -2f^2 - 2\epsilon df^2 + g'(t))H
\end{align*}
By the 1st variation formula together with the curvature assumption
$$\nabla_aHV_a = -\frac{1}{\delta}(A_{ab} + \epsilon Hg_{ab})V_aV_b \leq 0 \ \ \ \text{ and } \ \ \  H > 0$$
Therefore, we obtain the desired differential inequality if 
$$6\epsilon^2\delta H^2 - 4\epsilon f \leq 0$$
and
$$\frac{4fg}{\delta} - 2f^2 - 2\epsilon df^2 + g'(t) \geq 0$$
Recall
$$\delta = \frac{1}{1 + \epsilon d}, \ \ f(t) = \frac{1}{2t} + \frac{3}{2}\epsilon \delta \phi^2(\hat{T}), \ \ g(t) = \frac{1}{2t} + \frac{3\epsilon \delta}{4}\phi^2(\hat{T})$$
and $(p_0, t_0) \in M^d \times (0, \hat{T}]$ and $\phi(\hat{T}) = \sup_{M^d \times (0, \hat{T}]}H$. One can directly check that above two inequalities indeed hold, thus proving the desired differential inequality. \\

One can then use parabolic maximum principle as in \cite{hamilton1995harnack} to conclude that $$Z_{\epsilon, \hat{T}}(V) \geq 0$$
for $V \in TM^d$, $t\in (0, \hat{T}]$, hence prove theorem 1.1.  
\end{proof}

In the case of ancient solutions, the term $\frac{H}{2t}$ can be dropped.
\begin{corollary}[Harnack inequality for ancient solutions]
Let $\{M_t^d\}_{t \in (-\infty, 0)}$ be an ancient solution to mean curvature flow which is either closed or complete with bounded curvature such that $$A + \epsilon_0 Hg \geq 0 \text{ in }M^d \times (-\infty, 0)$$
    for some $\epsilon_0 \geq 0$ and $$\phi(t) = \sup_{M^d \times (-\infty, t]}H < \infty$$ 
    for all $t \in (-\infty, 0)$. Then
     $$Z(V) = (A + \epsilon_0Hg)(V,V) + \frac{2}{1 + \epsilon_0d}\nabla H\cdot V + \frac{1}{(1 + \epsilon_0d)^2}\nabla_tH + \frac{\epsilon_0}{(1 + \epsilon_0d)^2}H^3 + \frac{3\epsilon_0}{4(1 + \epsilon_0d)}\phi^2(t)H \geq 0$$
    for all $V \in TM^d$, $t \in (-\infty, 0)$. 
\end{corollary}
\begin{proof}
 For each fixed $V \in TM^d$, $t \in (-\infty, 0)$, let $-\alpha < t < 0$ and consider $\{M_s^d\}_{s \in [-\alpha, 0)}$. Since $\phi(t) \geq \sup_{M^d \times (-\alpha, t]} H$, theorem 1.1 implies
$$Z(V) = (A + \epsilon_0Hg)(V,V) + \frac{2}{1 + \epsilon_0d}\nabla H\cdot V + \frac{H}{2(t + \alpha)} + \frac{1}{(1 + \epsilon_0 d)^2}\nabla_tH + \frac{\epsilon_0}{(1 + \epsilon_0d)^2}H^3 + \frac{3\epsilon_0}{4(1 + \epsilon_0d)}\phi(t)^2H \geq 0$$
for all $V \in TM^d$, $t \in (-\alpha, 0)$. Thus letting $\alpha \to \infty$, we obtain corollary 2.1.    \end{proof} 

As in \cite{hamilton1995harnack}, we can integrate the Harnack quantity along spacetime curves to compare $H$ at different points in spacetime.
\begin{corollary}
Let $\{M^d_t\}_{t \in [0, T)}$ be as in theorem 1.1. Then for any $p_1, p_2 \in M^d$ and $0 < t_1 < t_2 < T$
$$\frac{H(p_2, t_2)}{H(p_1, t_1)} \geq (\frac{t_1}{t_2})^{\frac{(1 + \epsilon_0d)^2}{2}}exp(-\frac{1 + \epsilon_0d}{4}\Delta - \epsilon_0(1 + \frac{3 + 3\epsilon_0d}{4})\phi^2(t_2)(t_2 - t_1))$$
where $\Delta = \inf_{\gamma}\int_{t_1}^{t_2}|\gamma'(t)|^2_{g(t)}dt$. The infimum is taken among all smooth curves $\gamma : [t_1, t_2] \to M^d$ with $\gamma(t_1) = p_1$ and $\gamma(t_2) = p_2$.
\end{corollary}
\begin{proof}
    Let $\gamma : [t_1, t_2] \to M^d$ be any smooth curve connecting $p_1$ and $p_2$. Dividing the Harnack inequality by $H$ and setting $V = \frac{1}{2(1 + \epsilon_0 d)}\gamma'(t)$, we have
$$\frac{1}{(1 + \epsilon_0d)^2}\frac{\nabla H}{H}\cdot \gamma'(t) + \frac{1}{(1 + \epsilon_0d)^2}\frac{\nabla_tH}{H} + \frac{1}{4(1 + \epsilon_0d)}|\gamma'(t)|^2_{g(t)} + \epsilon_0(\frac{1}{(1 + \epsilon_0d)^2} + \frac{3}{4(1 + \epsilon_0d)})\phi^2(t_2) + \frac{1}{2t} \geq 0$$
Therefore
$$\frac{d}{dt}\ln{H(\gamma(t), t)} = \frac{\nabla H}{H}\cdot \gamma'(t) + \frac{\nabla_tH}{H} \geq -\frac{1 + \epsilon_0d}{4}|\gamma'(t)|^2_{g(t)} -\epsilon_0(1 + \frac{3 + 3\epsilon_0d}{4})\phi^2(t_2) - \frac{(1 + \epsilon_0d)^2}{2t}$$
Integrating in $t$ and taking infimum among all possible curves $\gamma$ gives
$$\ln{\frac{H(p_2,t_2)}{H(p_1, t_1)}} \geq -\frac{1 + \epsilon_0d}{4}\Delta - \epsilon_0(1 + \frac{3 + 3\epsilon_0d}{4})\phi^2(t_2)(t_2 - t_1) + \frac{(1 + \epsilon_0d)^2}{2}\ln{\frac{t_1}{t_2}}$$
thus proving corollary 2.2.
\end{proof}
\section{Ancient solutions}
To prove theorem 1.3, we need two lemmas. The first lemma allows us to replace condition (3) of theorem 1.3 with an equivalent statement which is easier to deal with when doing rescaling argument. 
\begin{lemma}
    Let $\{M_t^d\}_{t \in (-\infty, 0)}$ be as in theorem 1.3. Let $\phi(t) = \sup_{M^d \times (-\infty, t]}H < \infty$. Then
    $$\limsup_{t \to -\infty}\phi(t)\textit{{\normalfont diam}}(M_t^d) < \infty \iff \limsup_{t \to -\infty}\max_{M^d \times \{t\}}H\textit{{\normalfont diam}}(M_t^d) < \infty $$
\end{lemma}
\begin{proof}
By definition of $\phi$,
$$\limsup_{t \to -\infty}\max_{M^d \times \{t\}}H \textit{{\normalfont diam}}(M_t^d) \leq \limsup_{t \to -\infty}\phi(t)\textit{{\normalfont diam}}(M_t^d)$$
hence we trivially have $\implies$ direction. \\

To prove the reverse direction, note that $t \to \textit{{\normalfont diam}}(M_t^d) = \max_{x,y \in M^d}|X_t(x) - X_t(y)|$ is nonincreasing in $t$. We give a brief sketch of why this is true. Since we are working with smooth solutions to the mean curvature flow, for each $t \in (-\infty, 0)$
$$\textit{{\normalfont diam}}(M_t^d) = \max_{x,y \in M^d}|X_t(x) - X_t(y)| > 0$$
This ensures that $$(-\infty, 0)\ni t \to \textit{{\normalfont diam}}(M_t^d)$$
is locally Lipchitz in $t$, and thus is differentiable for almost every $t$, with
$$\frac{d}{dt}\textit{{\normalfont diam}}(M_t^d) = \frac{\partial}{\partial t}|X_t(p) - X_t(q)| = \langle \frac{X_t(p) - X_t(q)}{\textit{{\normalfont diam}}(M_t^d)}, \vec{H}(p,t) - \vec{H}(q,t) \rangle$$
where $(p,q) \in M^d \times M^d$ are the two distinct points which achieves the diameter, $\vec{H}$ is the mean curvature vector and $\langle \cdot, \cdot \rangle$ is the standard inner product in $\mathbb{R}^{d+1}$. On the other hand, since $$|X_t(p) - X_t(q)| = \max_{(x,y) \in M^d \times M^d}|X_t(x) - X_t(y)|$$
the spatial Laplacian of $|X_t(\cdot) - X_t(\cdot)|$ at $(p,q)$ is nonpositive. This implies that
$$ 0 \geq \langle \frac{X_t(p) - X_t(q)}{\textit{{\normalfont diam}}(M_t^d)}, \vec{H}(p,t) - \vec{H}(q,t) \rangle + \frac{2d}{\textit{{\normalfont diam}}(M_t^d)} = \frac{d}{dt}\textit{{\normalfont diam}}(M_t^d) + \frac{2d}{\textit{{\normalfont diam}}(M_t^d)}$$
which proves that $\textit{{\normalfont diam}}(M_t^d)$ is nonincreasing in $t$. \\

Assume 
$$\limsup_{t \to -\infty}\max_{M^d \times \{t\}}H\textit{{\normalfont diam}}(M_t^d) \leq C < \infty$$
Then for each small enough $t$, we have
$$\max_{M^d \times \{s\}}H\textit{{\normalfont diam}}(M_s^d) < C+1$$
for all $s \leq t$.
For any $\epsilon > 0$, we can find $s \in (-\infty, t]$ so that $$\phi(t) - \epsilon < \max_{M^d \times \{s\}}H$$
Since $\textit{{\normalfont diam}}(M_t^d) \leq \textit{{\normalfont diam}}(M_s^d)$, we have
$$(\phi(t) - \epsilon)\textit{{\normalfont diam}}(M_t^d) \leq \max_{M^d \times \{s\}}H\textit{{\normalfont diam}}(M_s^d) \leq C+1$$
Thus letting $\epsilon \to 0$, we have
$$\phi(t)\textit{{\normalfont diam}}(M_t^d) \leq C+1$$
for all sufficiently small $t$, thus proving that
$$\limsup_{t \to -\infty}\phi(t)\textit{{\normalfont diam}}(M^d_t) \leq C+1 < \infty$$
thus proving the lemma.    
\end{proof}

The 2nd lemma essentially follows from \cite{haslhofer2017mean}, proposition A.1 and \cite{chow2020extrinsic}, theorem 12.24. However, we include the proof here for the convenience of the reader, and to explain why our extra conclusion when additionally assuming $A > 0$ follows from their results.
\begin{lemma}
Let $X : M^d \times (-T, 0] \to \mathbb{R}^{d+1}$ be a smooth, complete, strictly mean convex solution to mean curvature flow of dimension $d \geq 2$ defined for some nonpositive time $t \in (-T, 0]$ for some $0 < T \leq \infty$. Suppose there exists $p_0 \in M^d$ so that
$$\frac{\lambda_1 + .. + \lambda_{d-1}}{H}$$
attains a minimum at $(p_0,0)$. Then there exists $\delta \in (0, T)$ so that 
$$A \geq 0$$
for all $M^d \times (-\delta, 0]$. In particular, the minimum of $\frac{\lambda_1 + .. + \lambda_{d-1}}{H}$ is nonnegative. Moreover, if $A > 0$ at $t = 0$, then $$A = \frac{1}{d}Hg$$
for all $M^d \times (-\delta, 0]$, 
hence $M^d$ is closed and $X_0 = X(\cdot, 0) : M^d \to \mathbb{R}^{d+1}$ is an isometric embedding with its image being a round sphere.
\end{lemma}
\begin{proof}
Define a tensor
$$M_{ij} = Hg_{ij} - A_{ij}$$
where $A$ is the 2nd fundamental form. Notice that $M$ and $A$ are simultaneously diagonalizable and if $V$ is the eigenvector of $A$ corresponding to $\lambda_i$, then it is also the eigenvector of $M$ corresponding to $\sum_{j \neq i}\lambda_j$.\\

Let $$\beta = \frac{\lambda_1 + .. + \lambda_{d-1}}{H}(p_0,0)$$
and consider $N = M - \beta Hg$. The assumption then implies that the symmetric (0,2) tensor $N$ is nonnegative in $M^d \times (-T, 0]$ and has a nontrivial kernel at $(p_0,0)$. Note that $N$ satisfies the following evolution equation
$$(\nabla_t - \Delta)N_{ab} = |A|^2N_{ab}$$
Therefore, the strong maximum principle (\cite{chow2007ricci}, theorem 12.50) applied to $N$ implies that \\\\
(i) There exists $\delta > 0$ so that 
$$E(t) = \cup_{x \in M^d}E_{(x,t)} = \cup_{x \in M^d}\textit{{\normalfont ker}}(N_{(x,t)}) \leq TM$$
is a smooth vector bundle of rank $l > 0$ for all $t \in (-\delta, 0)$.\\\\
(ii) The bundle $E(t)$ is invariant under parallel transport w.r.t $g(t)$ for all $t \in (-\delta, 0)$.\\

Then for each $t \in (-\delta, 0)$, $p \in M^d$, one can locally isometrically split the manifold $(M^d, g(t))$ near $p$ into a direct product $(U \times V, g_1 \oplus g_2)$ where $U$ is the integral manifold of $E(t)$ and $V$ is the integral manifold of $E(t)^{\perp}$ where $\perp$ is taken w.r.t $g(t)$. \\

Let $e_i = e_i(p,t)$ be the eigenvector of $A$ corresponding to $\lambda_i$ at $(p,t)$. Then because $e_{d}(p,t)$ is the 1st eigenvector of $N$, it must be contained in $E(t)$, hence
$$\beta = \frac{\lambda_1 + .. + \lambda_{d-1}}{H}(p,t)$$
Therefore, $e_i$ is an eigenvector of $N$ corresponding to eigenvalue 
$$\sum_{j \neq i}\lambda_j - \beta H = \lambda_d - \lambda_i $$
Thus either $e_i \in E(t)$ or $e_i \in E(t)^{\perp}$. \\

If $e_i \in E(t)$, then $\lambda_i = \lambda_d$. If $e_i \in E(t)^{\perp}$, then by the isometric splitting, 
$$\lambda_i\lambda_d = \textit{{\normalfont sec}}(e_i, e_d) = 0$$
hence $\lambda_i = 0$. Here, $\textit{{\normalfont sec}}(e_i, e_d)$ is the sectional curvature in directions $e_i$ and $e_d$. This together with the fact that $E(t)$ has rank $l > 0$ implies that at each $(p,t) \in M^d \times (-\delta, 0)$, the eigenvalues of $A_{(p,t)}$ are either $0$ or $\lambda_d$, with the corresponding eigenspaces having dimension $d - l$ and $l$ respectively. Thus for any possible $l > 0$, $A \geq 0$ in $M^d \times (-\delta, 0)$. Then by continuity, $A \geq 0$ up to time $t = 0$ as well, thus proving the first assertion.\\

If we additionally assume $A > 0$ at $t = 0$, then $l = d$. If not, for each $p \in M^d$, since $A_{(p,t)}$ has a nontrivial kernel for all $t \in (-\delta, 0)$, we can find $v(p,t) \in T_pM^d$ so that
$$A_{(p,t)}(v(p,t), v(p,t)) = 0 \ \ \ \text{ and } \ \ \ |v(p,t)|_{g(0)} = 1$$
Then we can find a sequence $t_j \to 0$ and $v_0 = v_0(p) \in T_pM^d$ so that
$$v(p, t_j) \to v_0 \text{ in }T_pM^d$$
In particular, $|v_0|_{g(0)} = 1$ as well. We also have
$$A_{(p,0)}(v_0, v_0) = \lim_{j \to \infty}A_{(p, t_j)}(v(p, t_j), v(p, t_j)) = 0$$
This contradicts the fact that $A > 0$ at $t = 0$, hence $l = d$.
This means that all the eigenvalues of $A$ are equal to $\lambda_d$. Therefore we have the relation
$$A = \frac{1}{d}Hg$$
for all $M^d \times (-\delta, 0]$.
Since $d \geq 2$, the Codazzi equation implies that 
$$\nabla A = 0$$
and since $H > 0$, we immediately see that $M^d$ is closed by Myers theorem, and the principal curvatures of $X_0 : M^d \to \mathbb{R}^{d+1}$ are all fixed positive constant, which implies that $X_0$ is an embedding with its image being a round sphere of radius $\frac{d}{H(0)}$ with
$$H(0) = H(q,0)$$
for any $q \in M^d$.    
\end{proof}

We now prove theorem 1.3 and theorem 1.4.
\begin{proof}[Proof of theorem 1.3]
Trivially $(1)$ implies $(2) - (5)$. By \cite{10.4310/jdg/1442364652}, theorem 1.1, $(2)$ implies $(1)$. We claim that either $(4)$ or $(5)$ implies $(3)$. By the weak maximum principle, it is easy to see that for all sufficiently small $t$
$$\min_{M^d \times \{t\}}H \leq \frac{C}{\sqrt{-t}}$$
for some fixed constant $C > 0$. Thus $(4)$ implies $(5)$. Once we assume type I curvature decay, $\frac{C_0}{\sqrt{-t}} \leq \phi(t) \leq \frac{C_1}{\sqrt{-t}}$
for some fixed constant $C_1 > C_0 >  0$ for all $t \leq T_0$ for some fixed $T_0 \in (-\infty, 0)$.  Then by the evolution equation
$$|X(p,t) - X(p, T_0)| \leq \int_{t}^{T_0}\frac{C_1}{\sqrt{-s}}ds \leq C_2\sqrt{-t}$$
for all $t \leq T_0$ and $p \in M^d$. Therefore
$$\textit{{\normalfont diam}}(M_t^d) \leq 2C_2\sqrt{-t} + \textit{{\normalfont diam}}(M_{T_0}^d) \leq C_3\sqrt{-t} \leq \frac{C_4}{\phi(t)}$$
for all sufficiently small $t$. Thus
$$\limsup_{t \to -\infty}\phi(t)\textit{{\normalfont diam}}(M_t^d) < \infty$$
hence by lemma 3.1, we have (3). \\

We now claim that $(3)$ implies that our solution is actually locally uniformly convex. Our proof follows that of \cite{10.4310/jdg/1442364652}, theorem 4.2. First note that 
$$\liminf_{t \to -\infty}\min_{M^d \times \{t\}}\frac{\lambda_1 + .. + \lambda_k}{H} > -\infty$$
for some $1 \leq k \leq d-1$ implies
$$\liminf_{t \to -\infty}\min_{M^d \times \{t\}}\frac{\lambda_1 + .. + \lambda_{d-1}}{H} > -\infty$$
Since
$$\lambda_i \geq \frac{\lambda_1 + .. + \lambda_k}{k}$$
for all $k < i \leq d-1$, one has
$$\liminf_{t \to -\infty}\min_{M^d \times \{t\}}\frac{\lambda_1 + .. + \lambda_{d-1}}{H} \geq \frac{d-1}{k}(\liminf_{t \to -\infty}\min_{M^d \times \{t\}}\frac{\lambda_1 + .. + \lambda_k}{H}) > -\infty$$
Therefore, it is enough to prove local uniform convexity assuming $k = d-1$. \\

We next claim that 
$$t \to \min_{M^d\times \{t\}}\frac{\lambda_1 + .. + \lambda_{d-1}}{H}$$
is nondecreasing in $t$. As in the proof of lemma 3.2, for any $a \in \mathbb{R}$ the tensor
$M = aHg - A$ satisfies the evolution equation
$$(\nabla_{t} - \Delta)M_{ab} = |A|^2M_{ab}$$
Now let $t_1 < t_2$ and let 
$$a = 1 - \min_{M^d \times \{t_1\}}\frac{\lambda_1 + .. + \lambda_{d-1}}{H}$$
Then at $t_1$, we see that $M \geq 0$. Since $(M^d, g(t))$ is complete with bounded curvature, we can apply weak maximum principle to conclude that
$$M \geq 0$$
for all $t \geq t_1$. This implies that $$Hg - A \geq (\min_{M^d \times \{t_1\}}\frac{\lambda_1 + .. + \lambda_{d-1}}{H})Hg$$
at $t = t_2$. Since 
$$\frac{\lambda_1 + .. + \lambda_{d-1}}{H}$$
is the smallest eigenvalue of $Hg - A$, we conclude that
$$\min_{M^d \times \{t_2\}}\frac{\lambda_{1} + .. + \lambda_{d-1}}{H} \geq \min_{M^d \times \{t_1\}}\frac{\lambda_1 + .. + \lambda_{d-1}}{H}$$
which proves the claim. \\

Set
$$\beta = \liminf_{t \to -\infty}\min_{M^d \times \{t\}}\frac{\lambda_1 + .. + \lambda_{d-1}}{H}$$
Then the claim implies that
$$\min_{M^d \times \{t\}}\frac{\lambda_1 + .. + \lambda_{d-1}}{H} \geq \beta$$
for all $t \in (-\infty, 0)$, or equivalently
$$Hg - A -\beta Hg \geq 0 \text{ in }M^d \times (-\infty, 0)$$
and we can find $t_j \to -\infty$ and $p_j \in M^d$ so that
$$\frac{\lambda_1 + .. + \lambda_{d-1}}{H}(p_j, t_j) \to \beta$$
We define rescaled flows 
$$X_j(p,t) = \phi(t_j)(X(p, t_j + \phi^{-2}(t_j)t) - X(p_j, t_j))$$
defined for $t \in (-\infty, 0]$. By the definition of $\phi(t_j)$, the mean curvature of the rescaled flows are uniformly bounded by 1. Therefore the rescaled flows have uniform bounds on all higher order derivatives of the curvature due to the condition $(1 - \beta)Hg - A \geq 0 $ and Ecker-Huisken interior gradient estimate. Also, by our definition of $X_j$, the rescaled flows all pass the origin at $t = 0$. Therefore, we can extract a limit flow $X_{\infty} : M^d_{\infty} \times (-\infty, 0] \to \mathbb{R}^{d+1}$ where $X_j \to X_{\infty}$ smoothly on compact subsets of $M^{d}_{\infty} \times (-\infty, 0]$. The limit flow $X_{\infty}$ is a complete, weakly mean convex solution to mean curvature flow. We claim that it is actually strictly mean convex. By the strong maximum principle, either $H_{\infty}$ is never zero or identically zero. In the latter case, since $$(1 -\beta)H_{\infty}g_{\infty} - A_{\infty}\geq 0$$
we have $$A_{\infty} \leq 0$$
Since $H_{\infty} = 0$ is the trace of a nonpositive tensor $A_{\infty}$, the only possibility is when $A_{\infty} = 0$. Combined with completeness, this implies that the limit flow is a static hyperplane in $\mathbb{R}^{d+1}$. On the other hand, by lemma 3.1, the time slices of the rescaled flows at $t = 0$ are all contained in a fixed ball of finite radius in $\mathbb{R}^{d+1}$ for all sufficiently large $j$. This implies that the time slice of the limit flow at $t = 0$ is also uniformly bounded in $\mathbb{R}^{d+1}$ which is a contradiction. Thus $H_{\infty} > 0$ everywhere, and the quotient $\frac{\lambda_1 + .. + \lambda_{d-1}}{H}$ passes through the limit. Therefore on the limit flow, we have the curvature condition
$$\frac{\lambda_1 + .. + \lambda_{d-1}}{H} \geq \beta$$
and by our definition of $(p_j, t_j)$, there exists a point $p_{\infty} \in M_{\infty}^d$ so that
$$\frac{\lambda_1 + .. + \lambda_{d-1}}{H}(p_{\infty}, 0) = \beta$$
This means $\beta = \frac{\lambda_1 + .. + \lambda_{d-1}}{H}(p_{\infty}, 0)$ is a spacetime interior minimum of $\frac{\lambda_1 + .. + \lambda_{d-1}}{H}$. Then lemma 3.2 implies that
$$A_{\infty} \geq 0$$
on $M_{\infty}^d \times (-\delta, 0]$
for some $\delta > 0$. \\

We claim that $A_{\infty} > 0$ in $M_{\infty}^d \times \{0\}$. If $A_{\infty}$ has a nontrivial kernel at $t = 0$, then the strong maximum principle combined with completeness implies that the flow globally isometrically splits off a line. In particular, the time slices of the limit flow are
unbounded. This once again contradicts the fact that the time slices of the rescaled flows at $t = 0$ are uniformly bounded for all sufficiently large $j$. Thus $A_{\infty} > 0$ when $t = 0$. \\

Then lemma 3.2 implies that $M_{\infty}^d$ is closed and $X_{\infty}(\cdot, 0)$ is an isometric embedding into $\mathbb{R}^{d+1}$ with its image being a round sphere. In particular, the local smooth convergence at $t = 0$ is actually global smooth convergence, hence the rescaled second fundamental form $\phi(t_j)^{-1}A(\cdot, t_j)$ converge uniformly smoothly to the second fundamental form of a round sphere of fixed radius as $j \to \infty$. Since a sphere is uniformly strictly convex, there exists $j_0$ so that for all $j \geq j_0$
$$A(\cdot, t_j) \geq 0$$
This implies that
$$\limsup_{t \to -\infty}\min_{M^d \times\{t\}}\frac{\lambda_1}{H} \geq 0$$
By applying the maximum principle to $A$ as we did for $M_{ab}$, we see that 
$$t \to \min_{M^d \times \{t\}}\frac{\lambda_1}{H}$$
is nondecreasing in $t$, thus
$$\liminf_{t \to -\infty}\min_{M^d \times \{t\}}\frac{\lambda_1}{H} = \limsup_{t \to -\infty}\min_{M^d \times \{t\}}\frac{\lambda_1}{H} \geq 0$$
Thus the original flow is weakly convex, and since it is closed, it is locally uniformly convex by strong maximum principle applied to $A$. Also, because one always has a lower bound
$$\frac{1}{\sqrt{-t}} \leq C\phi(t)$$
for some fixed $C > 0$ for all sufficiently small $t$, by lemma 3.1
$$\limsup_{t \to -\infty}\frac{\textit{{\normalfont diam}}(M_t^d)}{\sqrt{-t}} \leq C\limsup_{t \to -\infty}\phi(t)\textit{{\normalfont diam}}(M_t^d) < \infty$$
thus by \cite{10.4310/jdg/1442364652}, theorem 1.1, $\{M_t^d\}_{t\in (-\infty, 0)}$ is a family of shrinking round spheres. Therefore either (3), (4) or (5) implies (1), thus completing the proof of theorem 1.3. \end{proof}
\begin{proof}[Proof of theorem 1.4]
Assume otherwise. Then 
$$-\epsilon_0 = \liminf_{t \to -\infty}\inf_{M^d \times \{t\}}\frac{\lambda_1 + .. + \lambda_{k}}{H} < 0 $$
As in the proof of theorem 1.3, both
$$t \to \inf_{M^d \times \{t\}}\frac{\lambda_1 + .. + \lambda_k}{H}$$
is nondecreasing in $t$ when either $k = 1$ or $k = d-1$. Therefore
$$\inf_{M^d \times \{t\}}\frac{\lambda_1 + .. + \lambda_{k}}{H} \geq -\epsilon_0$$
and we can find $t_j \to -\infty$ and $p_j \in M^d$ so that
$$\frac{\lambda_1 + .. + \lambda_{k}}{H}(p_j, t_j) \to -\epsilon_0$$
We define rescaled flows 
$$X_j(p,t) = \phi(t_j)(X(p, t_j + \phi^{-2}(t_j)t) - X(p_j, t_j))$$
defined for $t \in (-\infty, 0]$. Then as in the proof of theorem 1.3, we can extract a limit flow $X_{\infty} : M^d_{\infty} \times (-\infty, 0] \to \mathbb{R}^{d+1}$ where $X_j \to X_{\infty}$ smoothly on compact subsets of $M^{d}_{\infty} \times (-\infty, 0]$. Also the `flatness at low curvature region' condition implies that
$$\frac{H(p_j, t_j)}{\phi(t_j)} \geq \delta_0 > 0$$
    for some fixed positive constant $\delta_0$ which corresponds to the $\delta$ in the statement of the flatness condition when $\epsilon = \frac{\epsilon_0}{2}$. This ensures that the limit is nonflat. Thus, the limit flow $X_{\infty} :M^d_{\infty} \times (-\infty, 0] \to \mathbb{R}^{d+1}$ is a complete, strictly mean convex solution to mean curvature flow. Then the quotient $\frac{\lambda_1 + .. + \lambda_{k}}{H}$ also passes through the limit, hence on the limit flow, we have the curvature condition
$$\frac{\lambda_1 + .. + \lambda_{k}}{H} \geq -\epsilon_0$$
and by our definition of $(p_j, t_j)$, there exists a point $p_{\infty} \in M_{\infty}^d$ so that
$$\frac{\lambda_1 + .. + \lambda_{k}}{H}(p_{\infty}, 0) = -\epsilon_0$$
This means $-\epsilon_0 = \frac{\lambda_1 + .. + \lambda_{k}}{H}(p_{\infty}, 0)$ is a strictly negative spacetime interior minimum of $\frac{\lambda_1 + .. + \lambda_k}{H}$. In the case when $k = 1$, this contradicts \cite{haslhofer2017mean}, proposition A.1, and when $k = d-1$, this contradicts lemma 3.2. Thus in any case, $-\epsilon_0 \geq 0$, thus we have weak $k$ convexity when $k = 1$ or $k = d-1$, thus proving theorem 1.4.    
\end{proof}
\begin{remark}
   The extra `flatness' condition is only used to prevent the limit flow from being a static hyperplane. Also, note that the scaling factor $\phi(t_j)$ is chosen to be able to extract a smooth limit. Therefore, as long as we can find some other conditions which ensures smooth convergence to a nonflat limit by possibly taking different scales, the argument holds without extra modification.
\end{remark}

\bibliographystyle{plain}
\bibliography{ref}
\end{document}